\documentclass[12pt]{amsart}
\usepackage{latexsym,amssymb, verbatim}

\theoremstyle{plain}
  \newtheorem{theorem}{Theorem}[section]
  \newtheorem{proposition}[theorem]{Proposition}
  \newtheorem{lemma}[theorem]{Lemma}
  \newtheorem{corollary}[theorem]{Corollary}
  
\theoremstyle{definition}
  \newtheorem{definition}[theorem]{Definition}
  \newtheorem{example}[theorem]{Example}

 \theoremstyle{remark}
  \newtheorem{remark}[theorem]{Remark}

\numberwithin{equation}{section}

\def\complexes{{\mathbb C}}

\def\im{\mathrm{im}}

\def\rank{\mathrm{rank}}

\def\Lyndon{\mathrm{Lyndon}}

\begin{document}

\title
[Symmetric chain decompositions for quotient posets]
{Symmetric chain decomposition for cyclic quotients of Boolean algebras and 
relation to cyclic crystals}

\author[P. Hersh]{Patricia Hersh}
\address{Department of Mathematics, North Carolina State University, Raleigh, NC 27695-8205}
\email{plhersh@ncsu.edu}

\author[A. Schilling]{Anne Schilling}
\address{Department of Mathematics, UC Davis, One Shields Ave., Davis, CA 95616-8633}
\email{anne@math.ucdavis.edu}

\thanks{The authors were supported by NSF grants 
DMS--1002636  and DMS--0652641, DMS--0652652, DMS--1001256, respectively.
}

\begin{abstract}
The quotient of a Boolean algebra by a cyclic group is proven to have a symmetric chain decomposition.  
This generalizes earlier work of Griggs, Killian and Savage on the case of prime order, giving an explicit
construction for any order, prime or composite. The combinatorial map specifying how to proceed downward 
in a symmetric chain is shown to be a natural cyclic analogue of the $\mathfrak{sl}_2$ lowering operator in
the theory of crystal bases.
\end{abstract}

\maketitle

\section{Introduction}

Griggs, Killian, and Savage~\cite{GKS} have given a symmetric chain decomposition for the 
quotient of a Boolean algebra $B_n$ by a cyclic group $C_n$ in the case where $n$ is prime.
Their paper was inspired by the search of symmetric Venn diagrams by showing a beautiful connection
to the Boolean lattice up to cyclic invariance, whose elements are called necklaces. Their construction 
introduces necklace representatives for the symmetric chain decomposition of Greene and Kleitman~\cite{GK} for 
$B_n$. There exist accessible surveys on these connections in~\cite{RSW,WW}.
We generalize the symmetric chain decomposition for $B_n/C_n$ to all $n$ through the use of Lyndon 
words~\cite{Re} and the theory of crystal bases~\cite{HK:2002}.

Our construction is based on a completely explicit injective map $\phi $ which takes a quotient 
poset element to an element it covers, and which may  be interpreted as a cyclic analogue of the 
lowering operator for $\mathfrak{sl}_2$ crystals. Two elements $x,y$ in the quotient poset
$B_n/C_n$ belong to the same symmetrically placed chain in our symmetric chain
decomposition if and only if  either $x = \phi^r (y)$ or 
$y = \phi^r (x)$ for some $r$.  

We use properties of Lyndon words (see for example~\cite{Re}) to define $\phi $ directly on the quotient poset and
prove it is injective. The definition of $\phi$ is then rephrased in a simpler way that
makes it evident that $\phi$ is a cyclic analogue of the usual  
$\mathfrak{sl}_2$ lowering operator
for crystals.  This rephrasing involves a matching of 0's with 1's and then proceeding downward
through a symmetrically placed chain by changing unmatched 1's to 0's in a Lyndon word from
right to left. It was inspired by the approach of~\cite{GKS}.
  
A different proof  for all $n$ of the existence of a symmetric chain decomposition (SCD) was given prior to ours 
by Jordan in~\cite{KKJ}.  Specifically, she provided an algorithm to produce an SCD for any $n$.  
The resulting SCD is  not the same one we obtain.  
Jordan's result was subsequently generalized in~\cite{DMT:2011} to show that
$B_n/G$ is a symmetric chain order whenever $G$ is generated by powers of disjoint cycles
in the symmetric group. Dhand~\cite{Dh} has also done related work.
Our work was done independently 
of~\cite{Dh}, \cite{DMT:2011} and~\cite{KKJ}; our approach is in fact quite different than is taken in these papers.  

One reason for interest in finding symmetric chain decompositions 
stems back to work of Stanley (cf.~\cite{Sta0, Sta2}) and Proctor (cf.~\cite{Proctor}), 
among others, in the 1980's on unimodality questions for rank
generating functions -- since a symmetric chain decomposition implies unimodality. 
In 1977,  Griggs gave quite general sufficient conditions in~\cite{Gr} for a poset to admit a symmetric chain 
decomposition by a clever argument involving the max flow-min cut theorem as well as Hall's marriage theorem.

Subsequently, a powerful approach to unimodality questions  was developed by using
the representation theory of $\mathfrak{sl}_2$, since dimensions of weight spaces of the 
same parity in any
$\mathfrak{sl}_2$-representation necessarily form a unimodal sequence.  
In this paper we show that this is also possible in the case of 
the quotient of a Boolean algebra by a cyclic group using crystal bases.

\subsection*{Acknowledgments}
We are indebted to Dennis Stanton for very helpful discussions.
We would also like to thank Georgia Benkart for her insight on unimodality of
crystals and the reduction to the $\mathfrak{sl}_2$ case. Thanks to Georgia Benkart, Monica Vazirani,
Stephanie van Willigenburg, and the Banff International Research Station for providing a stimulating
environment during the Algebraic Combinatorixx workshop, enabling the connection
to be made there between symmetric chain decomposition and crystal lowering
operators.

\section{Symmetric chain decomposition}

We regard the elements of the quotient poset $B_n/C_n$
as equivalence classes of words in $\{ 0,1\}^n$, where two words are 
equivalent if they differ by a cyclic shift of the position indices.  
We often speak of words,
when in fact we always mean equivalence classes of words.  Our poset is graded with
rank function being the number of 1's in a word.  We have $u\le v$ if and only if
there exist cyclic
rearrangements of $u$ and $v$, denoted $\sigma_i(u)$ and $\sigma_j(v)$, 
such that the set of positions with value 1 in $\sigma_i(u)$ is a 
subset of that of $\sigma_j(v)$. Here $\sigma_i$ denotes the rotation which increases indices 
by $i$ mod $n$.

\begin{example}
We have $01001\le 11110$ because $\sigma_1(11110) = 01111$.
\end{example}

Our plan is to define a completely explicit 
map $\phi $ that will tell us how to proceed downward through 
each chain in the symmetric chain decomposition. 

\begin{definition}
Given an ordered alphabet $\mathcal{A}$ and a word in $w\in \mathcal{A}^n$, define
the {\it Lyndon rearrangement} of $w$ to be the lexicographically smallest word obtained
by a cyclic rotation of the letters in $w$.  The resulting word is called a {\it Lyndon 
word}.
\end{definition}

\begin{example}
We use the alphabet $\mathcal{A} = \{ 0, 1\} $ with ordering $1\prec 0$.  Then  $00011101001100111001$ 
has $11101001100111001000$ as its Lyndon rearrangement.
\end{example} 

Now let us define the map $\phi $ which has as its domain all elements of the top half of
the quotient poset and those elements from the bottom half of the poset which do not turn
out to be bottom elements of chains in the symmetric chain decomposition.  The map 
$\phi $ sends  any poset element $x$ upon which it acts 
to an element $\phi (x)$ which  $x$ covers.
We focus first on the case of words with more 1's than 0's, i.e. the top half of the quotient
poset; we will handle the case of words
with at least as many 0's as 1's in a different manner that will depend on our map on the 
top half of ranks in the quotient poset.

Given a word of 0's and 1's, cyclically rotate it into its Lyndon rearrangement.
Now apply the following process repeatedly until all unmatched elements are 1's: take any
0 that is immediately followed by a 1 (cyclically) and 
match these pairs of letters, removing them from 
further consideration.  Using parentheses to depict our pairing, here is an example of 
this algorithm: $1101100110 \rightarrow 1)1(01)10(01)1(0 \rightarrow 1)1(01)1(0(01)1)(0$.

Call each of these 01 pairs  a {\it parenthesization pair}, or sometimes a {\it matching pair}.  
If there is at least one unpaired 1 at the end of this process, then $\phi $ maps the word 
to a word in which the rightmost unmatched 1 in the Lyndon expression is changed to a 0.
For example, $\phi $ maps $1)1(01)1(0(01)1)(0$ to $1101000110$.  Notice that the 
Lyndon rearrangement of this new word is $1101101000$ and this has bracketing
$1)1)(01)1)(01)(0(0(0$.  In particular, changing the 1 to a 0 created a new parenthesization
pair involving the letter changed from a 1 to a 0 together with another 1 which had been 
unmatched, provided there are at least two unmatched 1's just prior to the application of 
$\phi $.  We refer to  this pair of  letters
 the most recently created parenthesization pair.  If there is only a single
unmatched 1, this is still changed by $\phi $ to a 0, now yielding a word with a single 
unmatched letter, with it now being a 0.

Define this map $\phi $ on the lower half of ranks 
by successively undoing the most recently created parenthesization pair by turning its
letter that is still a 1 into a 0; for this map to be well-defined, we will need to prove
that for each element
$x$ in the lower half of the poset there is a unique chain
leading to it from the top half of the poset 
by successive application of $\phi $, since the definition of $\phi $ on the lower half
depends on this chain leading to $x$ from above.  This will be accomplished by
an induction on $n - \rank(x)$  once we have completed the proof of 
injectivity for $\phi $ on the top half of ranks.   

\begin{example}
Repeated application of $\phi $ to $1111011001011110000$ yields
the symmetric chain
\begin{equation*}
\begin{array}{cccccccccccccccccccc}
	1&1&1&1&0&1&1&0&0&1&0&1&1&1&1&0&0&0&0&\rightarrow \\
	)&)&)&)&(&)&&(&(&)&(&)&)&&&(&(&(&(
\end{array}
\end{equation*}
\vspace{.1in}
\begin{equation*}
\begin{array}{cccccccccccccccccccc}
		1&1&1&1&0&1&1&0&0&1&0&1&1&1&0&0&0&0&0&\rightarrow \\
	)&)&)&)&(&)&)&(&(&)&(&)&)&&(&(&(&(&(
\end{array}
\end{equation*}
\vspace{.1in}
\begin{equation*}
\begin{array}{cccccccccccccccccccc}
	1&1&1&1&0&1&1&0&0&1&0&1&1&0&0&0&0&0&0&\rightarrow \\
	)&)&)&)&(&)&)&(&(&)&(&)&)&&(&(&(&(&(
\end{array}
\end{equation*}
\vspace{.1in}
\begin{equation*}
\begin{array}{cccccccccccccccccccc}
	1&1&1&1&0&1&0&0&0&1&0&1&1&0&0&0&0&0&0&\; . \\
	)&)&)&)&(&)&&(&(&)&(&)&)&&&(&(&(&(
\end{array}
\end{equation*}
\vspace{.1in}

\end{example}

\begin{definition} 
Let $a_i,a_j,a_k,a_l$ be letters in a word $w$. We can match two such letters by drawing an
arc or edge between them. We say that the edge from $a_i$ to $a_k$ {\it crosses}
the edge from $a_j$ to $a_l$ if we have $i < j< k < l$ in some cyclic rearrangement of $w$.  
A parenthesization pair $\{ a_i, a_k\} $ {\it crosses} another pair if the edges between them cross.
\end{definition}

\begin{remark}
Notice that the set of parenthesization pairs in a word is noncrossing and invariant
under cyclic rotation.
\end{remark}

In fact, we will want the following related property:

\begin{lemma}\label{non-cross-lemma}
We cannot have a pair of unmatched 1's that crosses a parenthesization pair.
\end{lemma}

\begin{proof}
One endpoint of the matching edge would be a 0, while the other end would be a 1.
If this edge were crossing with some unmatched pair of 1's, then the  0 in the 
parenthesization pair would prefer to be matched with one of these
1's   than with its current matching partner, a contradiction to our parenthesization 
pair construction.
\end{proof}

Thus, all unmatched 1's are in the same connected component once noncrossing arcs are 
inserted connecting elements of parenthesization pairs (arranged around a circle).

\begin{corollary}
The map $\phi $ applied to any element in the upper half of ranks will not change the 
collection of parenthesization pairs other than creating one 
additional pair comprised of the letter transformed by $\phi $ from a 1 to a 0 along with 
the leftmost 1 that is unmatched just prior to this application of $\phi $.
\end{corollary}

\begin{proof}
This follows from Lemma ~\ref{non-cross-lemma}.
\end{proof}

In the next proof, we use the notation
$b_1\dots b_r \le_{\Lyndon} c_1\dots c_r$ to denote the fact
that $b_1\dots b_r$ is at least as small 
lexicographically as $c_1\dots c_r$.

\begin{proposition}
The map $\phi $ is injective on the  upper half of ranks.
\end{proposition}

\begin{proof}
Suppose  two different  cyclic
words $A = a_1\ldots a_n$ and $A' = a_1'\ldots a_n'$ 
satisfy $\phi(A) = \phi (A') = B$, with $A,A'$ each having more
1's than 0's.  Without loss of generality, we may
assume $A$ and $A'$
agree except that $a_i=1,a_i'=0,a_j=0,a_j'=1$ for some $i<j$ where 
$B$ agrees with both except that $B$ has $b_i=b_j=0$. 
That is, $\phi $ changes $a_i$ from 1 to 0 in $A$, $\phi $ changes $a_j'$ from 1 to 0 in 
$A'$.  We may also assume that
one of the words $A,A'$ is written as a Lyndon word, so that either $a_i$ or $a_j'$ is the
rightmost unmatched 1 in its Lyndon word.  That is, we have written one of the words
$A$ or $A'$ using its Lyndon rearrangement, and we have written the other in its 
rearrangement so that the words agree except in the two locations where one or the 
other word has a 1 switched to a 0 by $\phi $. We do not know a priori which of the two
words is written as a Lyndon word, but rather we will consider both cases.

Our plan is to show in both cases that 
all of the letters in the word $W(A) = W(A') = a_{i+1}\dots a_{j-1} $ 
belong to parenthesization pairs with other letters also in this segment,
implying $A'$ has $a_i'$ belonging to a parenthesization pair together 
with $a_j'$, a contradiction to $\phi $ acting on $A'$ by changing $a_j'$ to a 0.  

Suppose first that
$A$ is Lyndon.  If there is
an unmatched 1 in $W(A)$, this would contradict our definition of $\phi $: $\phi $
turns the rightmost
unmatched 1 into a 0 in the Lyndon rearrangement of $A$, since $a_i$ would then not
be rightmost.  There cannot be a 1 in $W(A)$ that is matched with a 0 outside $W(A)$, 
since this 0 outside $W(A)$ would instead match with $a_i$. 
On the other hand, if there were a 0 in $W(A')$ that is not matched
with any 1 in $W(A')$, this 
would match with $a_j'$, contradicting 
$a_j'$ being unmatched.  Thus, $W(A')$ must be fully matched, implying
$a_i'$ matches with $a_j'$, as desired. 

Now suppose instead that $A'$ is Lyndon.  Again, there cannot be any 0's in $W(A')$ that
are not matched with 1's in $W(A')$, since any such 0 would match with 
$a_j'$.  Suppose there is some $a_k'=1$ 
in $W(A')$ that is not matched
within $W(A')$.  Then the corresponding 1 in $W(A)$, i.e. $a_k=1$,
is also unmatched in $A$, hence must appear 
farther to the left than $a_i$ in the Lyndon rearrangement of $A$.   Thus, the Lyndon
rearrangement of $A$ would begin with a letter $a_s$ to the right of $a_i$ and to the left of 
$a_{k+1}$, whereas the Lyndon rearrangement of $A'$ begins with $a_1'$.  But then we would 
have $a_s\dots a_j \le_{\Lyndon} a_1 \dots a_{j-s+1} \le_{\Lyndon} a_1'\dots a_{j-s+1}'
\le_{\Lyndon} a_s'\dots a_j' = a_s\dots a_{j-1}a_j' <_\Lyndon a_s\dots a_j$, a contradiction.
Thus, each 1 in $W(A')$ is also matched within $W(A')$, so $W(A')$ is fully matched, implying 
that $a_i'$ again comprises a parenthesization pair with $a_j'$, again  a contradiction.
\end{proof}

For $n$ odd, the map at the middle pair of ranks is perhaps worth special discussion:

\begin{proposition}
For $n$ odd, the map $\phi $ is injective between the two middle ranks, i.e. when it
turns the unique unmatched 1 into a 0.
\end{proposition}

\begin{proof}
This is immediate from the fact that changing this 1 to a 0 cannot impact the 
collection of parenthesization pairs, so that we may look at an element of $\im(\phi )$ 
having one more 0 than 1 and determine its inverse by seeing which 0 is unmatched.
\end{proof} 

\begin{proposition}
The map $\phi $ is injective on the lower half of ranks.
\end{proposition}

\begin{proof}
To prove injectivity on the bottom half, we show how to define the inverse map.  Given
an element $u$ with at least as many 0's as 1's, 
we begin by finding the unique element $v$  of the 
same chain that is exactly as many ranks above the middle as $u$ is below the middle.
This is done as follows.  Obtain a parenthesization of $u$ by pairing any 0 immediately
followed by a 1 cyclically, and keep repeating until all unpaired letters are 0's.  Now 
obtain $v$ from $u$ by changing all these unpaired letters from 0's to 1's.  If the 
difference in ranks between $v$ and $u$ is $r$, then $\phi^r (v)=u$ and 
$\phi^{-1}(u) = \phi^{r-1}(v)$.  By induction, we may assume $\phi $ is well-defined at 
all ranks strictly above $u$, enabling us thereby to use the uniqueness of $v$ to 
prove that $\phi^{-1}(u)$ is unique.  This also 
enables us to apply $\phi $ in a well-defined manner to $u$.
\end{proof}

We now state our main result:

\begin{theorem}
The quotient poset $B_n/C_n$ has a symmetric chain decomposition such that 
$u,v\in B_n/C_n$ with $\rank(u) < \rank (v) $ 
belong to the same chain if and only if $u = \phi^r (v)$ for some $r$.
\end{theorem}

\begin{proof}
The fact that the chains are symmetric is immediate from the construction.
We use rank symmetry of the poset to deduce that we have fully covered the bottom half
of the poset from the fact that we fully covered the top half.
\end{proof}

\section{Alternative description of $\phi $}

Next we develop some properties of the map $\phi $ that yield a much more
explicit alternative description of $\phi $, enabling us to establish a connection in
the next section to
the theory of crystal bases.  The upshot will be that in spite of 
our taking the lexicographically 
smallest representation of each element to which $\phi $ applies in a symmetric 
chain, we nonetheless end up proceeding down a symmetric chain  
simply by changing the unmatched 1's to 0's from right to left 
in the Lyndon expression for the highest element of the symmetric chain.  

\begin{proposition}\label{move-right}
Each time the map $\phi $ is applied turning a 1 into a 0 in a Lyndon word
$a_1\dots a_n$ to obtain a word $a_1'\dots a_n'$ with $a_i=1,a_i'=0$ and $a_j=a_j'$ for
$j\ne i$, the Lyndon expression for
$a_1'\dots a_n'$ will either 
shift the letter $a_i'$ to the right or leave its position unchanged.
\label{move-0-right}
\end{proposition}

\begin{proof}
Suppose there is some $1<j<i$ such that the Lyndon expression for $a_1'\dots a_n'$ is
$a_j'\dots a_i'\dots a_n'a_1'\dots a_{j-1}'$.   Then 
$a_j'\dots a_i' \le_{\Lyndon} a_1'\dots a_{i-j+1}'  = a_1\dots a_{i-j+1} \le_{\Lyndon}
a_j\dots a_i <_{\Lyndon} a_j'\dots a_i'$, a contradiction.
\end{proof}

\begin{corollary}\label{nested-paren}
The series of parentheses added by successively applying the map $\phi $
in the top half of a Boolean algebra must be nested with respect to each other, with each successive step 
adding a new innermost parenthesis.
\end{corollary}

\begin{proof}
Proposition~\ref{move-0-right} implies that each time we turn a 1 into a 0 in the top half, creating a 
new parenthesized pair, the Lyndon rearrangement may only
move this pair farther to the right.  In particular, any remaining unmatched 1's must still be to
the left of this new 0 and to the right of its partner 1, yielding the desired nesting property.
\end{proof}

\begin{corollary}
The map $\phi $ proceeds from right to left through the initial Lyndon word, 
successively turning each 1 that is not initially paired with a 0 into a 0.
\end{corollary}

\begin{proof}
This follows from Proposition~\ref{move-right}. In particular, the letters that comprise the left
parentheses of the pairs created while proceeding downward through the top half of 
the Boolean algebra only move to the right at least until after any such letter is switched 
from a 1 to a 0.  Therefore, a letter that is the rightmost unmatched 1 will continue to be
this under the Lyndon rearrangement in the top half of ranks. In addition, the nesting of parenthesis pairs 
as justified in Corollary~\ref{nested-paren} ensures the result for the bottom half of ranks.
\end{proof}

\section{Relation to crystals}

We now point out the resemblance of the parenthesization procedure of the previous section and the
signature rule for crystal bases, see for example~\cite{HK:2002}. Combinatorially,
$\mathfrak{sl}_2$ crystals can be viewed as the graph on words of finite length in the
letters $\{1,2\}$ with an arrow from word $w$ to word $w'$ if $w'$ is obtained from $w$ by changing the
rightmost unbracketed $1$ into a $2$, where now we successively bracket pairs $21$.
Identifying $2$ with $0$ from the previous section, this is precisely the parenthesization (up to the cyclic shift).
If there is an arrow from $w$ to $w'$ in the crystal graph, the Kashiwara lowering operator $f$ acts as
$f(w)=w'$. If there is no arrow from $w$, $f$ annihilates $w$. Similarly, the Kashiwara raising operator
$e$ acts as $e(w')=w$ or annihilates $w'$ if there is no incoming arrow to $w'$.

An element $x$ (resp. $y$) in the crystal is called highest (resp. lowest) weight if $e(x)=0$ (resp. $f(y)=0$).
The weight of an element is the number of $1$'s in the word minus the number of $2$'s.
If the weight of the highest weight element is $L$, then the crystal is $L+1$ dimensional.

There exists an involution on the crystal, called the Sch\"utzenberger involution in type $A$
or Lusztig involution for general types, which interchanges highest and lowest weight vectors $x$ and $y$
and the lowering and raising operators $f$ and $e$ (in our $\mathfrak{sl}_2$ setting).

Interpreting the map $\phi$ from the previous section as the Kashiwara lowering operator of a cyclic crystal,
we can characterize the highest weight elements in the cyclic crystal.

\begin{lemma}
Highest weight elements in the cyclic $\mathfrak{sl}_2$ crystal are Lyndon words where, if we do the cyclic bracketing,
only the last string of consecutive $0$'s is bracketed cyclically, i.e. with $1$'s at the beginning of the word. Furthermore, 
removing a cyclic bracket by turning the corresponding $0$ to $1$ would break the Lyndon condition.
\end{lemma}

\begin{proof}
By Proposition~\ref{move-0-right} and Corollary~\ref{nested-paren}, the letters after application of $\phi$ move to the 
right after the rearrangement into a Lyndon word, and the parentheses that are added are nested. This proves the claim.
\end{proof}

\begin{example}
The word $1111001110001110$ is highest weight in the cyclic crystal, since the parenthesization yields:
\begin{equation*}
\begin{array}{cccccccccccccccc}
	1&1&1&1&0&0&1&1&1&0&0&0&1&1&1&0\\
	)&&&&(&(&)&)&&(&(&(&)&)&)&(
\end{array}
\end{equation*}
Changing the last 0 to a 1 would break the Lyndon condition.

The word $11100010110$ on the other hand is not highest weight since in
\begin{equation*}
\begin{array}{ccccccccccc}
	1&1&1&0&0&0&1&0&1&1&0\\
 	)&)&&(&(&(&)&(&)&)&(
\end{array}
\end{equation*}
the open bracket in position 4 and the closed bracket in position 2 are matched, which
go across the cycle where the 0 is not at the end.
In fact, this word would belong to the string of highest weight element
\begin{equation*}
\begin{array}{ccccccccccc}
1&1&1&1&0&0&1&0&1&1&0\\
 )&&&&(&(&)&(&)&)&(
\end{array}
\end{equation*}

The following word is not highest weight, since one can remove a cyclic bracket without
breaking the Lyndon property
\begin{equation*}
\begin{array}{cccccccc}
1&1&1&1&0&1&0&0\\
 )&)&&&(&)&(&(
\end{array}
\end{equation*}
The corresponding highest weight element is $11110110$.
\end{example}

\begin{corollary}\mbox{}
\begin{enumerate}
\item
The map $\phi$ is a natural cyclic analogue of the lowering operators from the theory of
crystal bases.
\item
Lusztig's involution sends elements in the upper (resp. lower) half of a
symmetrically placed chain to elements in the lower (resp. upper) half of a chain in the dual symmetric chain decomposition.
Notice that these two symmetric chain decompositions which are obtained from
each other via the Lusztig involution are typically different from each other.
\end{enumerate}
\end{corollary}

Our symmetric chain decomposition from the first part now yields the following:

\begin{theorem}
The $\mathfrak{sl}_2$ strings in the cyclic crystal yield a symmetric chain decomposition.
\end{theorem}


\end{document}